\documentclass[12pt,a4paper]{article}
\usepackage{amsthm,amsfonts,amsmath,amssymb,latexsym,amscd}
\usepackage{epsfig,graphics,color}
\usepackage[all]{xy}
\usepackage{verbatim}
\usepackage{bm}
\usepackage{enumerate}

\title{Contact orderability up to conjugation}
\author{Kai~Cieliebak$^1$\thanks{Partially supported by  DFG grant CI 45/8-1}\and Yakov~Eliashberg$^2$\thanks{Partially supported by NSF grant  DMS-1505910}\and Leonid Polterovich$^3$\thanks{Partially supported   by the ISF grant 178/13 and
the ERC  Advanced grant 338809}
\and \\
 \small{$^1$Institut f\"ur Mathematik, Universit\"at Augsburg, Germany}\\
 \small{$^2$Department of Mathematics, Stanford University, USA}\\
\small{ $^3$Department of Mathematics, Tel Aviv University, Israel
 }}

 \date{\today}
\theoremstyle{plain}
\newtheorem{theorem}{Theorem}[section]
\newtheorem{thm}[theorem]{Theorem}
\newtheorem{corollary}[theorem]{Corollary}
\newtheorem{cor}[theorem]{Corollary}
\newtheorem{proposition}[theorem]{Proposition}
\newtheorem{prop}[theorem]{Proposition}
\newtheorem{lemma}[theorem]{Lemma}

\theoremstyle{remark}

\newtheorem{remark}[theorem]{Remark}
\newtheorem{rem}[theorem]{Remark}
\newtheorem{example}[theorem]{Example}

\newtheorem{definition}{Definition}

\newcommand{\id}{{\rm id}}

\newcommand{\p}{\partial}

\newcommand{\om}{\omega}

\newcommand{\eps}{\varepsilon}
\newcommand{\into}{\hookrightarrow}
\newcommand{\la}{\langle}
\newcommand{\ra}{\rangle}
\newcommand{\wt}{\widetilde}
\newcommand{\wh}{\widehat}
\newcommand{\N}{{\mathbb{N}}}
\newcommand{\Z}{{\mathbb{Z}}}
\newcommand{\R}{{\mathbb{R}}}
\newcommand{\C}{{\mathbb{C}}}

\newcommand{\st}{{\rm st}}

\newcommand{\Ad}{{\rm Ad}}

\renewcommand{\min}{{\rm min}}
\renewcommand{\max}{{\rm max}}

\newcommand{\Diff}{{\rm Diff}}

\newcommand{\supp}{{\rm supp}}

\newcommand{\Supp}{\mathrm{supp}}
\newcommand{\DD}{\mathcal{D}}
\newcommand{\EE}{\mathcal{E}}

\renewcommand{\AA}{\mathcal{A}}

\newcommand{\g}{{\mathfrak g}}         

\newcommand{\fC}{{\mathfrak C}}
\newcommand{\fc}{{\mathfrak c}}

\begin{document}

\maketitle
  \rightline{\small{\em To the memory of Vladimir Igorevich Arnold}}
\tableofcontents

\section{Introduction}

A partial order on  groups of contact diffeomorphisms was  introduced
in \cite{EP} as a contact analog of Hofer's geometry for groups of
Hamiltonian diffeomorphisms of symplectic manifolds.
In this paper we  begin studying   the  remnants of this order on the conjugacy
classes of contactomorphisms.   Our main interest  in this paper are non-compact contact manifolds, and more specifically   a special class of non-compact contact manifolds which we call {\em convex at infinity}, see Section \ref{sec:conv-inf} below. While orderability problems for {\em closed} manifolds have obvious answers on the level of Lie algebra of contact vector fields, the situation for {\em non-compact} manifolds   is quite subtle already on the Lie algebra level. Problems of this kind naturally arisen in   connection with constructions of contact structures in \cite{BEM}.
The goal of the paper is to   illustrate the arising phenomena  on a
restricted class of examples, leaving   a more general   study, both in the Lie algebra and the group cases,
to our forthcoming paper \cite{CEP2}.
\subsection{Groups of contactomorphisms and their Lie algebras}
Let   $(U,\xi)$ be a coorientable noncompact  contact manifold.  We
fix a contact form $\alpha$ for $\xi$ and denote by $R$ its   Reeb vector field. Let
$$
   G:=\Diff_c(U,\xi)
$$
be the identity component of the group of
contactomorphisms of $(U,\xi)$ with compact support.
The Lie algebra $\g$ of $G$, which consists of compactly supported
contact vector fields,  can be identified with the space
$C^\infty_c(U)$ of smooth functions with compact support by
associating to each function $K$ its {\em contact vector field}
$$
   Y_K=KR + Z_K,\qquad Z_K\in\xi,\quad (dK+i_{Z_K}d\alpha)|_\xi=0.
$$
Note that
$$
   dK(Z_K)=0,\qquad L_{Y_K}\alpha=dK(R)\alpha.
$$
Conversely, given a contact vector field $Y$ its contact Hamiltonian
is defined by the formula  $K(x)=\alpha(Y(x))$, $x\in U$. Let us
stress the point  that to identify the Lie algebra $\g$ with the
function space $C^\infty_c(U)$ one needs to fix a contact form.

The adjoint action of $\psi\in G$ on $K\in\g$ computes to
\begin{equation}\label{eq:ad-contact}
   \Ad_\psi K = (c_\psi K)\circ\psi^{-1},
\end{equation}
where $c_\psi:U\to\R$ is the positive function satisfying
$\psi^*\alpha = c_\psi\alpha$. The Lie bracket on $\g$ is given by
$$
   \{H,K\} = dK(X_H)-KdH(R).
$$
The Lie algebra carries a canonical partial order defined by $H\leq K$
if $H(x)\leq K(x)$ for all $x\in M$, which is $\Ad$-invariant by
equation~\eqref{eq:ad-contact}.

\subsection{Dominating positive cones}\label{sec:domin}
Denote by $\g^{\geq0}$ the cone in the Lie algebra
$\g\cong C^\infty_c(U)$ consisting of nonnegative functions.

\begin{definition}\label{def:dominating}
A subcone $\fc \subset \g^{\geq0}\setminus 0$ is called a {\em
  dominating (positive) cone} if the following hold:
\begin{enumerate}
\item $\fc$ is $\Ad$-invariant;
\item $\fc$ is relatively open in $\g^{\geq0}\setminus\{0\}$;
\item for each $H\in\g$ there exists $K\in\fc$ with $H\leq K$;
\item for all $H\in\g$, $K\in\fc$ there exist $t>0$ and $g\in G$ such
  that $t\Ad_gH\leq K$;
\item for each $H\in\g^{\geq0}\setminus\{0\}$ there exist $g_1,\dots,g_k\in G$ such that
  $\Ad_{g_1}H+\dots+\Ad_{g_k}H\in\fc$.
\end{enumerate}
\end{definition}

\begin{remark}
Property (v) is not needed in this paper, but will become relevant for
the discussion of partial orders on contactomorphism groups in~\cite{CEP2}.
\end{remark}

Clearly, if the manifold $U$ is closed then the only dominating cone
in $\g$ is the cone $\g^{>0}$ consisting of everywhere positive functions.
If $U$ is not closed, then a dominating cone in general  need not exist. For
instance, $S^1\times\R^2$ with the contact form $dt+\frac12(xdy-ydx)$ does not admit any
dominating cone because if $\supp(H)\supset S^1\times D_R$ and
$\supp(K)\subset S^1\times D_r$ with $r<R$, then there is no
contactomorphism $g\subset G$ such that $g(\supp(H))\subset \supp(K)$,
see~\cite{eliash-shapes}.

However, there is an important class of noncompact contact manifolds,
called {\em convex at infinity}, for which a dominating cone always
exists. We discuss this class in the next subsection.

\subsection{Contact manifolds convex at infinity}\label{sec:conv-inf}
A noncompact contact manifold $(U,\xi)$ is called {\em convex at infinity} if
there exists a contact embedding $\sigma:U\into U$ which is contactly
isotopic to the identity such that $\sigma(U)\Subset U$, i.e    $\sigma(U)$  has a compact closure in $U$. The space of
all embeddings $\sigma$ with this property will be denoted by
$$
   \EE=\EE(U,\xi).
$$
Note that by cutting off a contact isotopy, the
restriction $\sigma|_K$ to  any compact set $K\subset U$ can be
extended to a contactomorphism in the group $G=\Diff_c(U,\xi)$.

The notion of contact convexity for hypersurfaces  in a contact
manifolds was introduced in \cite{EG-convex} and studied in detail in
\cite{Giroux-convex}. Let us recall that a hypersurface in a contact
manifold is called {\em convex} if it admits a transverse contact vector field.
The coorientation of this vector field is irrelevant because if $Y$ is
contact then $-Y$ is contact as well.

\begin{example}\label{ex:convex-infinity}
\noindent (i) A major class of contact manifolds convex at infinity is
provided by interiors of compact manifolds with convex boundary.
Indeed, as the required embedding $\sigma$ one can take the flow for
small positive time of an inward pointing contact vector field
transverse to the boundary.

 \noindent (ii) More generally, suppose a contact manifold $(U,\xi)$
admits a (not necessarily complete) contact vector field $Y$ without
zeroes at infinity, which outside a compact set is gradient-like for
an exhausting function $\phi:U\to \R$. Then  $(U,\xi)$ is convex at
infinity.  Indeed, first  use \cite[Lemma 2.6]{Eliash-Weinstein-revisited} to conclude
that  for a sufficiently large $c$ the end  $(\{\phi\geq c\},\xi)$ is contactomorphic to $(\Sigma \times[0,\infty),\wh\xi)$ such that  the vector field $\frac{\p}{\p s}$ is contact. Here we set $\Sigma:=\{\phi=c\}$ and   denoted by $s$ the coordinate  corresponding to the second factor. There is a contact isotopy  $h_t:(\Sigma \times[0,\infty),\wh\xi)\to  (\Sigma \times[0,\infty),\wh\xi)$, $t\in[0,1]$, such that that $h_0=\id$, $h_t=\id$ near $\Sigma\times0$ for all $t\in[0,1]$ and $h_1(\Sigma\times n)=\Sigma\times\frac{n-1}n$, $n=1,\dots$, which implies that $h_1(\Sigma \times[0,\infty))=\Sigma\times[0,1)$.

 \noindent (iii)
In a $3$-dimensional contact manifold  a {\em generic surface is
  convex}, see \cite{Giroux-convex}, hence the interior of a
generic  connected compact contact manifold with non-empty boundary is
convex at infinity. If the boundary components of a 3-manifold are
2-spheres and the manifold is tight near the boundary, then it is
convex at infinity, even when the boundary components are not convex,
see \cite{eliash-20Martinet}.  

\noindent (iv) On the other hand, the contact manifold
$S^1\times \R^2=(\R/\Z)\times\R^2 $ with the tight contact form
$dt+\frac12(x\,dy-y\,dx)$ is not convex at infinity, see \cite{eliash-shapes}.
\end{example}

Now we will introduce our main example.
Let $\lambda_\st=\frac12\sum\limits_1^n(x_idy_i-y_idx_i)$ be the standard
Liouville form on $\R^{2n}$ with its  Liouville vector field
$$
   Z=\frac12\sum_{1}^n\Bigl(x_i\frac{\p}{\p x_i}+y_i\frac{\p}{\p y_i}\Bigr),
$$
and let $\alpha_\st=\lambda_\st|_{S^{2n-1}}$ be
the standard contact form on the unit sphere $S^{2n-1}\subset\R^{2n}$.
Let us order coordinates in $\R^{2n}$ as $(x_1,\dots, x_n,y_1,\dots,
y_n)$ and denote by $\Pi_{k}$   a $k$-dimensional coordinate subspace
of $\R^{2n}$ which is spanned by the last $k$   vectors of the basis. For instance,
$\Pi_1$ is the $y_n$-coordinate axis, while $\Pi_{2n-1}$ is the
hyperplane $\{x_1=0\}$. Note that $\Pi_k$ is isotropic when $k\leq n$,
and coisotropic otherwise. We denote by $\Pi_k^\perp$ the orthogonal
subspace spanned by the first $2n-k$ basic vectors.

\begin{lemma}\label{lem:main-ex}
(a) For each $k=1,\dots,2n-1$ the contact manifold $(S^{2n-1}\setminus
\Pi_k,\xi_\st)$ is convex at infinity. Moreover, it can be contracted
by an element of $\EE$ to an arbitrarily small neighborhood of the
equatorial sphere $S^{2n-1}\cap\Pi_k^\perp$.

(b) For $k\geq n$ the manifold $(S^{2n-1}\setminus\Pi_k,\xi_\st)$ is
contactomorphic to $J^1(S^{2n-k-1})\times \R^{2k-2n}=T^*(S^{2n-k-1}\times\R^{k-n})\times\R$.
\end{lemma}

\begin{proof}
Recall that $r=|x|^2+|y|^2$ induces the canonical isomorphism
$(\R^{2n}\setminus\{0\},\lambda_\st)\cong(\R_+\times
S^{2n-1},r\alpha_\st)$ under which the Liouville vector field
$Z$ corresponds to $r\frac{\p}{\p r}$. Thus contact vector fields on $S^{2n-1}$
are in one-to-one correspondence with Hamiltonian vector fields on
$\R^{2n}\setminus\{0\}$ which commute with $Z$. Note that each linear
vector field on $\R^{2n}$ automatically commutes with $Z$.

(a) First  consider  the case $k\leq n$. The linear vector field
$$
   \wh Y_k := \sum_{j=n-k+1}^n\Bigl(-x_j\frac{\p}{\p x_j}+y_j\frac{\p}{\p y_j}\Bigr)
$$
is the Hamiltonian vector field of the function
$\sum_{n-k+1}^nx_jy_j$. It commutes with $Z$, so it descends to a
contact vector field $Y_k$ to $S^{2n-1}$. On
$\R^{2n}\setminus(\Pi_k\cup\Pi_k^\perp)$ the field $\wh Y_k$ is
gradient-like for the $Z$-invariant function $-\ln
\Bigl(1-\frac{\sum_{n-k+1}^ny_j^2}{|x|^2+|y|^2}\Bigr)$, hence on
$S^{2n-1}\setminus(\Pi_k\cup\Pi_k^\perp)$ the field $Y_k$ has no
zeroes and is gradient-like for the exhausting function $-\ln
\Bigl(1-\sum_{n-m+1}^n y_j^2\Bigr)$. Now convexity at infinity
follows from Example~\ref{ex:convex-infinity}(ii), and the flow of $Y_k$
for very negative times contracts any compact set $S^{2n-1}\setminus \Pi_k$ to a
neighborhood of $S^{2n-1}\cap\Pi_k^\perp$.

The case $k>n$ follows from   part (b) and  Example~\ref{ex:convex-infinity}(ii).

(b) For $k \geq n$ let
$$
   \wh Z_k := \sum_{j=1}^{2n-k}\Bigl(x_j\frac{\p}{\p x_j}-y_j\frac{\p}{\p y_j}\Bigr)
$$
be  the Hamiltonian vector field of the function
$ -\sum\limits_{1}^{2n-k} x_jy_j$. It  descends to a complete contact vector field $Z_k$ on  $S^{2n-1}$. The flow of $Z_k$ contracts every compact set in $S^{2n-1}\setminus \Pi_k $ to a neighborhood of the isotropic sphere $S^{2k-1}\cap\Pi_k^\perp$ and the field $-Z_k$ is gradient like for the function $    \left(\sum_1^n y_j^2+\sum_{2n-k+1}^n x_j^2\right) $ on $S^{2n-1}\setminus  \Pi_k$.     By Weinstein-Darboux theorem contact structures
 on $S^{2n-1}\setminus \Pi_k $ and $J^1(S^{2n-k-1})\times \R^{k-n}$ are isomorphic on tubular neighborhoods of isotropic sphere  $S^{2n-k-1}=S^{2n-1}\cap\Pi_k^\perp$  and the $0$-section $S^{2n-k-1}\times 0\subset J^1(S^{2n-k-1})\times \R^{2k-2n}$. This isomorphism then extends to a contactomorphism between $S^{2n-1}\setminus \Pi_k $ and $J^1(S^{2n-k-1})\times \R^{2k-2n}$ by matching the corresponding trajectories of the contact  vector field
$-Z_k$  with trajectories of the canonical contact vector field on  $J^1(S^{2n-k-1})\times \R^{2k-2n}$ contracting this manifold to its $0$-section.

\end{proof}

\subsection{The maximal dominating cone $\g^+$}\label{sec:g+}
Let $(U,\xi)$ be a contact manifold convex at infinity.
\begin{lemma}\label{lem:dominating-cone}
For $(U,\xi)$ connected and convex at infinity the cone
$$
   \g^+ := \{H\in\g^{\geq0}\mid H|_{\sigma(U)}>0 \text{ for some}\; \sigma\in\EE(U,\xi)\}.
$$
is dominating and maximal, (i.e., all other dominating cones are subcones of $\g^+$).
\end{lemma}

\begin{proof}
Properties (i), (ii) and (iii) in Definition~\ref{def:dominating} are clear.

For (iv), consider $H\in\g$, $K\in\g^+$. Then $C:=\supp (H)$ is compact and $K$
is positive on $\sigma(U)$ for some $\sigma\in\EE(U,\xi)$.
By cutting off the contact isotopy from the identity to $\sigma$
outside $C$ we find $g\in G$ with $g|_C=\sigma|_C$. Then
$\supp(\Ad_gH)=g(\supp (H))=\sigma(C)\subset\sigma(U)$. Since
$K|_{\sigma(U)}>0$, it follows that $t\Ad_gH\leq K$ for $t$
sufficiently small.

For (v), let $H\in\g^{\geq0}\setminus\{0\}$ be strictly positive on some open set
$V\subset U$. Pick any $\sigma\in\EE(U,\xi)$. Since the group $G$ acts transitively on $U$
and $\sigma(U)$ is relatively compact, there exist $g_1,\dots,g_k\in
G$ such that $\sigma(U)\subset g_1(V)\cup\cdots\cup g_k(V)$. Then
$\Ad_{g_i}H$ is nonnegative and strictly positive on $g_i(V)$,
hence $\Ad_{g_1}H+\dots+\Ad_{g_k}H$ is strictly positive on $\sigma(U)$
and therefore belongs to $\g^+$.

To prove maximality of $\g^+$, let $\fc$ be any other dominating cone.
Take any  $H\in\fc$ and $F\in\g^+$. Then by
Definition~\ref{def:dominating} there exists $K\in\fc$ such that $F\leq K$,
and there exist $g\in G$ and $t>0$ such that $t\Ad_gK\leq H$. It
follows that $t\Ad_gF\leq H$. Since $F\in\g^+$, this implies that
$H$ is positive on $\sigma(U)$ for some $\sigma\in\EE$, and therefore $H\in\g^+$.
\end{proof}

\begin{lemma}[Examples of maximal dominating cones]\label{lm:convex-domin}
\hfill

(i) For the standard contact structure on $\R^{2n+1}$ we have $\g^+=\g^{\geq0}\setminus 0$.

(ii) For the $1$-jet space $U=J^1(M)$ of a closed manifold $M$
  endowed with its standard contact structure, the maximal dominating cone
  $\g^+$ consists of all nonnegative functions whose support contains a
  neighborhood of a Legendrian submanifold isotopic to the zero section.

(iii) For $(S^{2n-1}\setminus \Pi_k,\xi_\st)$ as in Example (iii) in
  Section \ref{sec:conv-inf}, the cone $\g^+$ consists of all nonnegative
  functions which are positive on an image of the equatorial
  sphere $S^{2n-1}\cap\Pi_k^\perp$ under a contactomorphism isotopic
  to the identity.

(iv) In the special case  $(S^{3}\setminus \Pi_1,\xi_\st)$, which is
  the same as $\R^3\setminus 0$  with the standard contact structure
  inherited from $\R^3$, the cone $\g^+$ can also be characterized as
  consisting of all nonnegative functions whose support contains a
 neighborhood  homologically non-trivial $2$-sphere.
\end{lemma}

\begin{proof} Both contact manifolds in (i) and (ii) admit  complete contact  vector fields which contract every compact subset to an arbitrarily small neighborhood of the origin in  case (i), and
 to an arbitrarily small  neighborhood of the
zero section in case (ii).  For $(\R^{2n-1},dt+\sum\limits_1^{n-1}(x_jdy_j-y_jdx_j))$ this is the vector field $-2\frac{\p}{\p t}-\sum\limits_1^{n-1}\left(x_j\frac{\p}{\p x_j}+y_j\frac{\p}{\p y_j}\right)$, and for $(J^1(M), dz+pdq) $ this is the vector field
$-\frac{\p}{\p z}-p\frac{\p}{\p p}$.
But the group $G$ acts transitively on points and
on Legendrian submanifolds isotopic to the zero section, respectively.
In (iii), according to Lemma~\ref{lem:main-ex} the space
$S^{2n-1}\setminus \Pi_k$ can be contracted by an element in $\EE$
to a neighborhood of the equatorial sphere $S^{2n-1}\cap\Pi_k^\perp$.
For (iv), we note in addition that any two smoothly isotopic $2$-spheres in a
tight contact manifold can be $C^0$-approximated by spheres which are
contactly isotopic, see \cite{eliash-20Martinet}.
\end{proof}

\begin{remark}
If $U$ contains a compact subset which is not  contractible in $U$, then the cone $\g^+$
never coincides with $\g^{\geq0}\setminus\{0\}$. To see this, pick
$H,K\in\g^{\geq0}\setminus\{0\}$ such that $\supp (H)$ is
noncontractible in $U$
and $\supp (K)$ is contractible in $U$. Suppose there exists
$g\in G$ and $t>0$ with $t\Ad_g H\leq K$. Then we must have
$g(\supp (H))\subset \supp (K)$, which is impossible if $\supp (H)$ is
noncontractible and $\supp (K)$ is contractible in $U$.
\end{remark}

\subsection{Partial order on $\g^+$ up to conjugation}\label{sec:partial-order}
Let us denote  by $\Theta:=\g^+/\sim$ the quotient space of $\g^+$ by the
adjoint action of $G$ on $\g$. The partial order $H\leq K$ on $\g^+$
descends to a possibly degenerate partial order $\preceq$ on
$\Theta$ defined on $h,k\in\Theta$ by
$$
   h\preceq k :\Longleftrightarrow \hbox {there exists}\; H\in h, K\in k\;\hbox{such that}\;
   H\leq K.
$$

\begin{lemma}\label{lem:non-orderable-g}
The following are equivalent:
\begin{enumerate}[(a)]
\item there exists $H\in\g^+$ and $g\in G$ such that $\Ad_gH\leq sH$ for some $0<s<1$;
\item for all $K_1,K_2\in\g^+$ there exists $h\in G$ such that
  $\Ad_hK_1\leq K_2$.
\end{enumerate}
\end{lemma}

\begin{proof}
Clearly (b) implies (a). Conversely, suppose that (a) holds for
elements $H$, $g$ and let
$K_1,K_2\in\g^+$ be given. By Definition~\ref{def:dominating} there
exist $t_i>0$ and $h_i\in G$ such that
$$
   t_1\Ad_{h_1}K_1\leq H\leq\frac{1}{t_2}\Ad_{h_2}K_2.
$$
Applying $\Ad_g^N$
for some $N\in\N$ to these inequalities, we obtain
$$
   t_1\Ad_g^N\Ad_{h_1}K_1\leq\Ad_g^NH\leq s^NH\leq\frac{s^N}{t_2}\Ad_{h_2}K_2.
$$
Applying $\Ad_{h_2}^{-1}$ to both sides and dividing by $t_1$, we obtain
$$
   \Ad_{h_2}^{-1}\Ad_g^N\Ad_{h_1}K_1\leq \frac{s^N}{t_1t_2}K_2.
$$
Hence $\Ad_hK_1\leq K_2$ with $h:=h_2^{-1}g^Nh_1$, provided that
$N$ is chosen so large that $s^N\leq t_1t_2$.
\end{proof}

We call the positive cone $\g^+$ {\em non-orderable up to
  conjugation} if the equivalent conditions in
Lemma~\ref{lem:non-orderable-g} hold, and {\em orderable up to conjugation}
otherwise. Thus to prove orderability up to conjugation of $\g^+$, it
suffices to find {\em some} pair $K_1,K_2\in\g^+$ for which there
exists no $h\in G$ with $\Ad_hK_1\leq K_2$.

\begin{remark}
a) Even if  $\g^+$ is orderable up to conjugation this does not imply
that the induced binary relation on $\Theta$ is a genuine
order. However, we do not know any counterexamples to this
implication. We will discuss the arising structures in more detail  in
Section~\ref{section-geometry} below.

\noindent b) If the manifold $U$ is closed then the cone $\g^+$ is always orderable up to
conjugation for the following trivial reason: the volume integral
$$
   I(H) := \int\limits_U \Bigl(\frac{\alpha}{H}\Bigr)\wedge
   d\Bigl(\frac{\alpha}{H}\Bigr)^{n-1}
$$
satisfies $I(\Ad_gH)=I(H)$ for all $g\in G$, so one can never have
$\Ad_gH\leq sH$ for some $0<s<1$. Note that   the   {\em strict order $H>G$} does descend in the case  of a closed $U$ to a genuine order on $\Theta$, as it follows from the same preservation of volume argument.
\end{remark}

\begin{proposition}\label{prop:non-order}
(a) If $(U,\xi)$ is the standard contact $\R^{2n+1}$ or $J^1(M)$,
as in Lemma~\ref{lm:convex-domin} (i) and (ii), then $\g^+$ is
non-orderable up to conjugation.

(b) More generally, let $(V,\lambda)$ be the completion of a Liouville
 domain (see \cite{CE12}). Then for its contactization
$\bigl(U=V\times\R,\ker(\lambda+dt)\bigr)$ the maximal dominating cone
$\g^+(U,\xi)$ is non-orderable up to conjugation.
\end{proposition}

\begin{proof}
Since (a) follows from (b), it suffices to prove (b).
The Liouville flow $\phi_s$ on $V$ induces a contact  diffeotopy   $\psi_s(x,t)=(\phi_s(x),e^st)$ of $V\times\R$
satisfying $\psi_s^*(\lambda+dt)=e^s(\lambda+dt)$. Let $C\subset V\times\R$ be the attractor of the flow $\psi_s$ when $s\to-\infty$.  Take $K_1,K_2\in\g^+$,  $K_1\geq K_2$. By the definition of the cone $\g^+$
there exists a   contacomorphism $h \in G$ such that $\Supp(\Ad_h K_2)$ contains a neighborhood of $C$.  The flow $\psi_s$ when $s\to-\infty$ moves $\Supp(K_1)$ into an arbitrarily small neighborhood of $C$, and hence for sufficiently large $-s$
we have $\Supp(\Ad_{\psi_s}(K_1)=\Supp(K_1\circ\psi_s^{-1})\subset \Supp(\Ad_h K_2)$. Therefore,
  $(\Ad_{\psi_s}\wt K_1)(x)=e^sK_1(\psi_s^{-1}(x))\leq \Ad_hK_2(x)$ for  for sufficiently large $-s$, or
   $(\Ad_{h^{-1}\circ\psi_s} K_1)(x) \leq  K_2(x)$, which means that     $\g^+$ satisfies condition (b) in
Lemma~\ref{lem:non-orderable-g}.
\end{proof}

Proposition~\ref{prop:non-order}(b) combined with
Lemma~\ref{lem:main-ex}(b) yields

\begin{cor}
If $k\geq n$, then for $(S^{2n-1}\setminus\Pi_k,\xi_\st)$ from
Section~\ref{sec:conv-inf} the cone $\g^+$ is
non-orderable up to conjugation.
\hfill$\square$
\end{cor}

By contrast, we will show below in Section~\ref{sec:capac}

\begin{theorem}\label{thm:order-algebra}
If $k<n$, then for $(S^{2n-1}\setminus\Pi_k,\xi_\st)$ the
cone $\g^+$ is orderable up to conjugation: there
exist $H,K\in\g^+$ for which there is no $g\in G$ with $\Ad_g H\leq K$.
More precisely, there exists a surjective map $w:\g^+\to(0,\infty)$
such that $w(\Ad_gH)=w(H)$ for $g\in G$, $w(sH)=s^{-2}w(H)$ for any $s>0$ and such that $H\leq K$ implies $w(H)\geq w(K)$.
\end{theorem}

\section{Orderability and symplectic non-squeezing}\label{sec:special}

In this section we will rephrase Theorem~\ref{thm:order-algebra} as a
non-squeezing result for suitable unbounded domains in the standard
symplectic space $(\R^{2n},\om=\sum_1^n dx_j\wedge dy_j)$. Throughout
this section we fix $k$ with $1\leq k\leq n$ and denote
$$
   U_k:=S^{2n-1}\setminus \Pi_k,\qquad G_k:=\Diff_c(U_k,\xi_\st).
$$
\subsection{The class $\fC_k$ of unbounded domains in $\R^{2n}$}
Introduce ``polar coordinates", $r=|x|^2+|y|^2\in\R$,
$\theta=r^{-1/2}(x,y)\in S^{2n-1}$, so that the standard Liouville
form $\lambda_\st$ can be written as
$$
   \lambda_\st=\frac12\sum_1^n(x_jdy_j-y_jdx_j) = r\alpha_\st,
$$
where $\alpha_\st$ is the standard contact form on the unit  sphere
$S^{2n-1}$ which defines the standard contact structure
$\xi_\st=\ker\alpha_\st$. The coordinates $(r,\theta)$ identify
$(\R^{2n}\setminus 0,\lambda_\st)$ with the symplectization
$(\R_+\times S^{2n-1}, r\alpha_\st)$ of the standard contact structure on $S^{2n-1}$.
Thus the symplectization of $U_k=S^{2n-1}\setminus \Pi_k$
gets identified with $\R^{2n}\setminus \Pi_k$.

Note that any contactomorphism $\phi$ of $(S^{2n-1},\xi_\st)$ defines
a symplectomorphism $S\phi:\R^{2n}\to\R^{2n}$, singular at the origin, by the formula
$$
   S\phi(r,\theta) := \Bigl(\frac{r}{c_\phi(\theta)},\phi(\theta)\Bigr),
$$
where $\phi^*\alpha_\st = c_\phi(\theta)\alpha_\st$.

\begin{lemma}\label{lem:smoothing}
If $\phi$ is contactly isotopic to the identity, then there exists a
constant $K_\phi>1$ such that for any $\eps>0$
there exists a smooth symplectomorphism $S_\eps\phi$ of $\R^{2n}$ which
equals the identity on the $\eps$-ball around the origin and which coincides
with $S\phi$ outside the $(K_\phi\eps)$-ball.
Moreover, if $\phi$ as well as
its isotopy to the identity equal the identity near some compact
subset $C\subset S^{2n-1}$, then $S_\eps\phi$ can be chosen equal to
the identity on the cone over $C$.
\end{lemma}

\begin{proof}
Let $\phi_t$, $t\in[0,1]$ be a contact isotopy of $(S^{2n-1},\xi_\st)$
connecting $\phi_0=\id$ to $\phi_1=\phi$ with $\phi_i=\id$ on $C$.
Let $S\phi_t:\C^n\setminus 0\to\C^n \setminus 0$ be its symplectization.
Set
$$
   M:=\mathop{\max}\limits_{t\in[0,1]}  \max (c_{\phi_t}),\qquad
   m:=\mathop{\min}\limits_{t\in[0,1]}  \min (c_{\phi_t})
$$
and note that $m\leq 1\leq M$. Then for any $\delta>0, t\in[0,1]$ we have
$$
   \left\{r\leq \frac{\delta}{M}\right\}\subset
   S\phi_t(\{r\leq\delta\})\subset\left\{r\leq\frac{\delta}{m}\right\}.
$$
Therefore, we can extend the Hamiltonian isotopy
$S\phi_t|_{\{r\geq\delta\}}$ to a Hamiltonian isotopy $g_t:\C^n\to
\C^n$ such that $g_0=\id$ and  for all $t\in[0,1]$
$$
   g_t|_{\left\{r\leq \frac{\delta}{2M}\right\}\cup\{\theta\in K\}}=\id,\qquad
   g_t|_{\left\{r\geq\frac{2\delta}m\right\}}=S\phi_t.
$$
Then for $\delta =2M\eps$  the symplectomorphism $S_\eps\phi:=g_1$ satisfies the required
conditions with $K_\phi=\frac{4M}m$.
\end{proof}

Given a nonnegative compactly supported contact Hamiltonian
$H:U_k\to\R$ we extend it by $0$ to $S^{2n-1}$ and will keep the
notation $H$ for this extension. Recall that according to Lemma \ref{lm:convex-domin} the cone
$\g^+=\g^+(U_k)$ consists of all functions satisfying the conditions
\begin{enumerate}
\item $H =0$ near $S^{2n-1}\cap\Pi_k=S^{2n-1}\cap \{x=0,y_1=\dots=y_{n-k}=0\}$;
\item $H$ is positive  on the  image $g(S^{2n-1}\cap\Pi_k^\perp)$ of the equator
under   a contactomorphism  $g\in G_k$.
\end{enumerate}

We now define a class $\fC_k$ of domains in $\R^{2n}$ which, in
particular (see Lemma \ref{lm:gtoC} below), contains all the domains of the form
$$
   V(H) := \{(r,\theta)\in\R_+\times S^{2n-1}\mid
   rH(\theta)<1\},\qquad H\in\g^+(U_k).
$$

First, we  add to $\fC_k$ all hyperboloids
$$
   V_k^{a,b} := \left\{\frac1{a^2}\left(\sum\limits_1^n
   x_i^2+\sum\limits_1^{n-k} y_i^2\right)
   -\frac1{b^2}\sum\limits_{n-k+1}^{n} y_i^2 <1\right\},\qquad a,b>0.
$$

Let $\DD^{a,b}_k$ denote the identity component of the group of
Hamiltonian diffeomorphisms of $\R^{2n}$ supported away from
$V^{a,b}_k$ and set
$$
   \DD_k := \bigcup\limits_{a,b>0}\DD_k^{a,b}.
$$
It follows from Lemma~\ref{lem:smoothing} above that for any
contactomorphism $\phi\in G_k$ the smoothed symplectomorphism
$S_\eps\phi:\R^{2n}\to\R^{2n}$ belongs to $\DD^{a,b}_k$ if $a$ and $\frac ab$ are
small enough. Moreover, $S_\eps\phi$ agrees with $S\phi$ outside
$V^{a',b'}_k$ for any $b'>0$ and $a'> K_\phi\eps$, where $K_\phi$ is
the constant from Lemma \ref{lem:smoothing}. Thus, although the
smoothing $S_\eps\phi$ is not canonical, its action on domains which contain
$V^{a',b'}_k$ with $a'>K_\phi\eps$ is independent of the choice of the smoothing.

Now we are ready to give the general definition of the domains which
form the class $\fC_k$.
\begin{definition}\label{def:fCk} A connected open domain $V\in\R^{2n}$ belongs to
$\fC_k$ if there exist $a_1,b_1,a_2,b_2>0$ and a symplectomorphism
$\Phi\in\DD_k$ such that
$$
   V_k^{a_1,b_1}\subset V\subset \Phi(V_k^{a_2,b_2}).
$$
\end{definition}
The group $\DD_k$, and hence the group $G_k$, acts on $\fC_k$ by symplectomorphisms.

\begin{lemma}\label{lm:gtoC}
(i) For $H\in\g^+(U_k)$ we have $V(H)\in\fC_k$. 

(ii) For $H\in\g^+(U_k)$, $\phi\in G_k$ and $\eps$ sufficiently small we have
$$
   S_\eps\phi\bigl(V(H)\bigr) = V(\Ad_\phi H).
$$
(iii) If $H,K\in\g^+(U_k)$ satisfy $H\geq K$, then $V(H)\subset V(K)$.
\end{lemma}

\begin{figure}[h]
\begin{center}
\includegraphics[scale=.66 ]{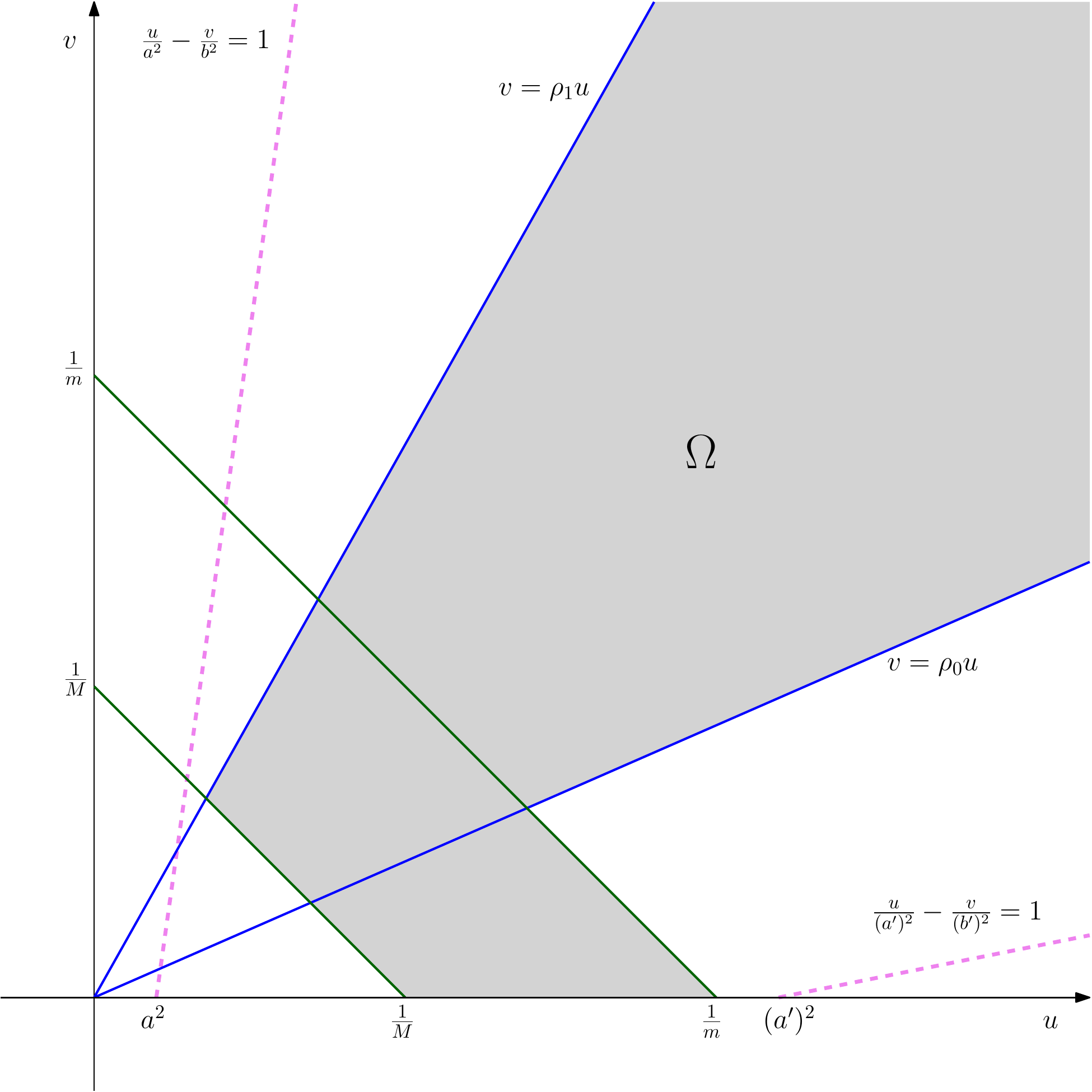}
\caption{Squeezing $V(H)$ between two hyperboloids}\label{fig:squeezing}
\end{center}
\end{figure}

\begin{proof}
Claims (ii) and (iii) are straightforward. To prove 
(i) we first observe that, since the class $\fC_k$ is invariant under
the action of the group $\DD_k$, we can replace $H$ by $\Ad_g(H)$ for any $g\in G_k$.
Hence we can assume without loss of generality that $H|_{S^{2n-1}\cap\Pi_k^\perp}>0$.
For $(x,y)\in\R^{2n}$ we denote 
$$
   u:=|x|^2+\sum_1^{n-k}y_j^2,\quad
   v:=\sum\limits_{n-k+1}^ny_j^2, \quad
   r=u+v=|x|^2+|y|^2,\quad \rho:=\frac vu.
$$
Then $\rho=\tan^2\alpha$, where $\alpha$ is the angle between the
vector $(x,y)\in\R^{2n}$ and the subspace $\Pi_k^\perp$. 
We can  view $\rho$ as a function on $S^{2n-1}$. Take $H\in\g^+$ and
set $M:=\max (H).$ Then $H|_{ \rho\geq\rho_1 }=0$ for a sufficiently
large $\rho_1$, and for a sufficiently small $\rho_0$ we have
$m:=\mathop{\min}\limits_{\rho\leq\rho_0}(H)>0$. Figure~\ref{fig:squeezing}
shows that
$$
   \p V(H)\subset\Omega:=\left\{ v\leq\rho_1u, u+v\geq\frac1M
   \right\}\setminus \left\{ v<\rho_0 u,   u+v >\frac1m  \right\}.
$$ 
Consider the hyperboloids 
$$
   V^{a,b}_k=\left\{ \frac u{a^2}-\frac v{b^2}<1\right\},\qquad
   V^{a',b'}_k=\left\{ \frac u{(a')^2}-\frac v{(b')^2}<1\right\}.
$$
An elementary geometric argument illustrated by Figure~\ref{fig:squeezing}
\footnote{We thank V.~Stojisavljevi\'c for preparing this figure.}
shows that
\begin{equation*} 
   \Omega\subset V^{a',b'}_k\setminus V^{a,b}_k,
\end{equation*}
 provided that
$$ 
   a^2 < T:= (M(1+\rho_1))^{-1},\qquad \left (\frac{b}{a}\right)^2 > \frac{\rho_1T}{T-a^2}
$$
and
$$ 
   \left(\frac{b'}{a'}\right)^2<\rho_0,\qquad (a')^2> \frac1m\;.
$$
These inequalities  express  the condition  that the dotted lines $ \frac u{(a)^2}-\frac v{(b)^2}=1$
  and $ \frac u{(a')^2}-\frac v{(b')^2}=1$, representing the boundaries of the domains
  $V^{a,b}_k$ and $V^{a',b'}_k$, do not intersect the shaded region representing $\Omega$.
 Thus, this  choice of $a,b,a', b'$ guarantees that
 $V_k^{a,b}\subset V(H)\subset  V_k^{a',b'} $, hence $V(H)\in\fC_k$.
\end{proof}
\subsection{Capacity-like function on $\fC_k$ and proof of
  Theorem~\ref{thm:order-algebra}}\label{sec:capac}
The following theorem will be proved in Section~\ref{sec:cap-proof} below.

\begin{thm}[Capacity]\label{thm:capacity}
There exists a capacity-like function $w:\fC_k\to(0,\infty)$ with the
following properties:
\begin{enumerate}
\item $w\bigl(\Psi(V)\bigr)=w(V)$ for all $\Psi\in \DD_k$ and $V\in\fC_k$;
\item $V\subset V'$  implies $w (V)\leq w(V')$ for all $V,V'\in\fC_k$;
\item $w(sV)=s^2w(V)$ for all $s>0$;
\item $w(V_k^{a,b})=\pi a^2$ for all $a,b>0$.
\end{enumerate}
\end{thm}

\begin{cor}\label{cor:non-squeezing2}
For any domain $V\in \fC_k$ and $s>1$ there is no $\Psi\in\DD_k$ such that $\Psi(sV)\subset V$.
\end{cor}

\begin{proof}
By Theorem~\ref{thm:capacity}(iii) we have
$$
   w(sV)=s^2w(V)>w(V),
$$
and the result follows from Theorem~\ref{thm:capacity}(i) and (ii).
\end{proof}

\begin{proof}[Proof of Theorem~\ref{thm:order-algebra}]
We define the  required function $w:\g^+\to\R$ by the formula
$w(K):=w\bigl(V(K)\bigr)$. For $g\in G_k$ we have
$V(\Ad_g(K))=S_g\phi(V(K))$, so Theorem \ref{thm:capacity}(i) implies that
the function $w$ is constant on orbits of the adjoint action.
We have $V(sK)=s^{-2}V(K)$, hence $w(sK)=s^{-2}w(K)$ in view of
Theorem~\ref{thm:capacity}(iii). This also yields surjectivity of
$w$, and Theorem~\ref{thm:capacity}(ii) implies that if $w(H)<w(K)$ then
there is no contactomorphism $g\in G$ such that $\Ad_gH\leq K$.
\end{proof}

 \section{Invariants of domains from $\fC_k$}\label{sec:invariants}

\subsection{Floer-Hofer symplectic homology of bounded domains in $\R^{2n}$}

Filtered symplectic homology $SH^{(a,b)}(U)$ of a bounded open set $U$
in the standard symplectic $\R^{2n}$ was introduced by A.~Floer and
H.~Hofer in~\cite{FH94} as a far-reaching generalization of earlier
symplectic invariants, such as Gromov's symplectic width~\cite{Gr85},
and later symplectic capacities, see~\cite{HZ94}. Since then the
invariant has been greatly generalized and expanded, but for the purposes
of this paper we will use the original Floer-Hofer version up to the
following slight modification. Instead of taking as in~\cite{FH94} a
direct limit over Hamiltonians which are negative on $U$ and equal a
positive definite quadratic form at infinity, we will take an inverse
limit over nonpositive Hamiltonians with compact support in $U$. This
version enjoys the same functorial properties as the one
in~\cite{FH94} but will be more convenient for the computations below. For
domains with smooth boundary of restricted contact type, our version
of symplectic homology differs from the one in~\cite{FH94} only by a
degree shift of $-1$ (this follows e.g.~from the duality results in~\cite{CO}).
We use $\Z/2$-coefficients and grade all groups by Conley-Zehnder index.

Let $\DD$ denote the group of (not necessarily compactly supported)
Hamiltonian diffeomorphisms of $\R^{2n}$.
The following proposition summarizes some relevant properties of
symplectic homology, see~\cite{FH94,FHW94}.

\begin{theorem}[Floer-Hofer]\label{thm:SH-original}
Filtered symplectic homology assigns to each bounded open subset
$U\subset\R^{2n}$ and numbers $0\leq a<b<\infty$ a $\Z$-graded
$\Z/2$-vector space $SH^{(a,b)}(U)$ with the following properties.

(Functoriality) Each $\Psi\in\DD$ induces isomorphisms
$$
   \Psi_*:SH^{(a,b)}(U) \stackrel{\cong}\longrightarrow SH^{(a,b)}\bigl(\Psi(U)\bigr).
$$
(Transfer map) Each inclusion $\iota:U\into V$ induces a homomorphism
$$
   \iota_!:SH^{(a,b)}(V)\to SH^{(a,b)}(U)
$$
It follows that for $\Psi\in\DD$ with $\Psi(U)\subset V$, the inclusion
$\iota:\Psi(U)\into V$ together with $\Psi$ induces a homomorphism
$$
   \Psi_!:=\Psi_*^{-1}\circ\iota_!:SH^{(a,b)}(V)\to SH^{(a,b)}(U).
$$
(Isotopy invariance) For a smooth family $\Psi^s\in\DD$ with
   $\Psi^s(U)\subset V$ for all $s\in[0,1]$, the maps
   $\Psi^s_!:SH^{(a,b)}(V)\to SH^{(a,b)}(U)$ are independent of $s$.

(Window increasing homomorphism) For $0\leq a<b$ and $0\leq a'<b'$ with $a\leq
   a'$ and $b\leq b'$ we have natural homomorphisms
$$
   SH^{(a,b)}(U) \to SH^{(a',b')}(U).
$$
(Symplectic homology of a ball)
The symplectic homology in the action window $(0,c)$ of the ball
$B_a^{2n}$ of radius $a$ in $\R^{2n}$ is given by
$$
   SH_j^{(0,c)}(B_a^{2n}) \cong\begin{cases}
      \Z/2, & c>\pi a^2\text{ and }j=n, \\
      \Z/2, & c>\pi a^2\text{ and }j=n\Bigl(2\lfloor\frac{c}{\pi a^2}\rfloor-1\Bigr)-1, \\
      0, & \text{otherwise}.
   \end{cases}
$$
\end{theorem}

\subsection{Symplectic homology for domains from $\fC_k$}

We extend the definition of symplectic homology to unbounded open
domains $V\subset\R^{2n}$ by
$$
   SH^{(a,b)}(V) := \mathop{\lim}\limits_{\longleftarrow}SH^{(a,b)}(U),
$$
where the inverse limit is taken over all bounded open subsets
$U\subset V$. We also define symplectic homology in the infinite
action window $(a,\infty)$ as
$$
   SH^{(a,\infty)}(V) := \mathop{\lim}\limits_{c\to\infty} SH^{(a,c)}(V).
$$
The extended symplectic homology still satisfies the properties in
Theorem~\ref{thm:SH-original}. However, the invariants one can
extract from these general properties are not sufficient for our
purposes. Instead, we will concentrate on the special class $\fC_k$ of unbounded
domains introduced in Section~\ref{sec:special} above and study their
invariants under the smaller group $\DD_k$ which preserves this class.

We begin with the computation of symplectic homology of the
hyperboloids $V_k^{a,b}$.

\begin{prop}\label{prop:computation}
Let $1\leq k<n$ and $0<c<\infty$. Then:

(a) For $a,b>0$ we have
$$
   SH_j^{(0,c)}(V_k^{a,b}) \cong\begin{cases}
      \Z/2, & c>\pi a^2\text{ and }j=n-k, \\
      \Z/2, & c>\pi a^2\text{ and }j=(n-k)\Bigl(2\lfloor\frac{c}{\pi a^2}\rfloor-1\Bigr)-1, \\
      0, & \text{otherwise}.
   \end{cases}
$$
(b) For $\wt a>a$, $\frac{\wt b}{\wt a}<\frac ba$ and $c>\pi\wt a^2$ the transfer map
$$
   \Z/2=SH_{n-k}^{(0,c)}(V_k^{\wt a,\wt b})\to  SH_{n-k}^{(0,c)}(V_k^{a,b})=\Z/2
$$
is an isomorphism.

(c) For $c>\pi a^2$ the window increasing homomorphism
$$
   \Z/2=SH_{n-k}^{(0,c)}(V_k^{a,b})\to  SH_{n-k}^{(0,\infty)}(V_k^{a,b})=\Z/2
$$
is an isomorphism.
\end{prop}

Heuristically, this result is easy to understand: All closed
characteristics on the boundary of $V_k^{a,b}$ satisfy
$z_{n-k+1}=\cdots=z_n=0$, so they agree with the closed
characteristics on the boundary of the ball $B_a^{2(n-k)}$ of radius
$a$ in $\R^{2(n-k)}$ whose symplectic homology is given in
Theorem~\ref{thm:SH-original}.
The actual proof of Proposition~\ref{prop:computation} will be given
in Section~\ref{sec:computation} below.

\begin{corollary}\label{cor:augmentation}
For any domain $V\in\fC_k$ there exists $c_0$ such that for each
$c\geq c_0$ and for all $a,b$ with $V^{a,b}_k\subset V$, the composition
$$
  \alpha_{V,c}:SH_{n-k}^{(0,c)}(V) \to SH_{n-k}^{(0,c)}(V_k^{a,b}) \to
  SH_{n-k}^{(0,\infty)}(V_k^{a,b})=\Z/2
$$
(where the first map is the transfer map and the second one
the window increasing homomorphism) is surjective.
\end{corollary}

\begin{proof}
Let $V\in\fC_k$ and $V^{a,b}_k\subset V$. By definition, there exist
$a_1,b_1,a_2,b_2>0$ and $\Phi\in\DD_k$ such that $V_k^{a_1,b_1}\subset V\subset \Phi(V_k^{a_2,b_2})$.
We choose $a_1$ so small that
$$
   V_k^{a_1,b_1}\subset V^{a,b}_k\subset V\subset \Phi(V_k^{a_2,b_2})
$$
and $\Phi_s=\id$ on $V_k^{a_1,b_1}$ for all $s\in[0,1]$, where
$\Phi_s$ is the isotopy in $\DD_k$ from the identity to $\Phi_1=\Phi$.
Set $c_0:=\pi a_2^2$ and consider $c>c_0$.
Then $\Phi_s^{-1}(V_k^{a_1,b_1})=V_k^{a_1,b_1}\subset V_k^{a_2,b_2}$,
so by isotopy invariance the map
$$
   (\Phi_s^{-1})_!:\Z/2=SH_{n-k}^{(0,c)}(V_k^{a_1,b_1})\to SH_{n-k}^{(0,c)}(V_k^{a_2,b_2})=\Z/2
$$
is independent of $s\in[0,1]$, hence an isomorphism by Proposition~\ref{prop:computation}(b).
It follows that the composition of the obvious maps
\begin{align*}
   \Z/2 &= SH_{n-k}^{(0,c)}(V_k^{a_2,b_2}) \cong SH_{n-k}^{(0,c)}\bigl(\Phi(V_k^{a_2,b_2})\bigr)
   \to SH_{n-k}^{(0,c)}(V) \cr &\to SH_{n-k}^{(0,c)}(V_k^{a,b}) \to SH_{n-k}^{(0,c)}(V_k^{a_1,b_1})=\Z/2
\end{align*}
is an isomorphism. This implies that
$SH_{n-k}^{(0,c)}(V_k^{a,b})=\Z/2$ and the map $SH_{n-k}^{(0,c)}(V)
\to SH_{n-k}^{(0,c)}(V_k^{a,b})$ is surjective, which combined with
Proposition~\ref{prop:computation}(c) proves the corollary.
\end{proof}

\subsection{A capacity for domains from $\fC_k$}
For any $c>0$ and $V\in\fC_k$ we define the {\em augmentation}
$$
  \alpha_{V,c}:SH_{n-k}^{(0,c)}(V) \to SH_{n-k}^{(0,\infty)}(V_k^{a,b})=\Z/2
$$
as in Corollary~\ref{cor:augmentation}, where $V^{a,b}_k\subset V$.
For another hyperboloid $V^{a',b'}_k\subset V$ we find a hyperboloid
$V^{a_1,b_1}_k\subset V^{a,b}_k\cap V^{a',b'}_k$, and the commuting diagram
\begin{equation*}
\xymatrixcolsep{6pc}
\xymatrix{SH^{(0,c)}_{n-k}(V) \ar[r] \ar[d] &
   SH^{(0,\infty)}_{n-k}(V^{a,b}_k)=\Z/2 \ar[d]^{\cong} \\
   \Z/2=SH^{(0,c)}_{n-k}(V_k^{a',b'})\ar[r]^{\cong} &
   SH^{(0,\infty)}_{n-k}(V_k^{a_1,b_1})=\Z/2 }.
\end{equation*}
shows that $\alpha_{V,c}$ does not depend on the choice of $a,b$.
Corollary~\ref{cor:augmentation} shows that $a_{V,c}$ is trivial for
sufficiently small $c$ and surjective for sufficiently large $c$.
The following corollary now is immediate from the properties of symplectic
homology in Theorem~\ref{thm:SH-original}.

\begin{cor}\label{cor:augm}
(i) For $\Psi\in\DD_k$ and $a,b$ such that $\Psi=\id$ on $V_k^{a,b}$
  we have the commutative diagram
\begin{equation*}
\xymatrixcolsep{6pc}
\xymatrix{SH^{(0,c)}_{n-k}(V) \ar[r]^{a_{V,c} } \ar[d]^{\Psi_*} &
   SH^{(0,\infty)}_{n-k}(V^{a,b}_k)=\Z/2 \ar[d]^{\id} \\
   SH^{(0,c)}_{n-k}\Bigl(\Psi(V)\Bigr)\ar[r]^{a_{\Psi(V),c}} &
   SH^{(0,\infty)}_{n-k}\Bigl(\Psi(V^{a,b}_k)=V^{a,b}_k\Bigr)=\Z/2\;.}
\end{equation*}
(ii) For an inclusion $\iota:V\into V'$ we have the commuting diagram
\begin{equation*}
\xymatrixcolsep{6pc}
\xymatrix{SH^{(0,c)}_{n-k}(V') \ar[r]^{a_{V',c} } \ar[d]^{\iota_!} &
   \Z/2 \ar[d]^{\id} \\
   SH^{(0,c)}_{n-k}(V)\ar[r]^{a_{V,c}} & \Z/2\;.}
\end{equation*}
\end{cor}

\subsection{Proof of Theorem~\ref{thm:capacity}}\label{sec:cap-proof}
As in~\cite{FHW94}, we define the capacity-like function $w:\fC_k\to(0,\infty]$ on $V\in\fC^k$ by
$$
   w(V) := \inf\{c \mid a_{V,c}:SH^{(0,c)}_{n-k}(V)\to \Z/2 \;\hbox{is
     surjective}\}.
$$
Let us verify that the function $w$ has all the properties in
Theorem~\ref{thm:capacity}. Properties (i) and (ii) follow from
the commuting diagrams in Corollary~\ref{cor:augm}, and (iv) follows
from Proposition~\ref{prop:computation}. Property (iii) follows as
in~\cite{FHW94}, noting that the conformal rescaling $z\mapsto sz$ on
$\R^{2n}$ preserves the class of domains $\fC_k$ and multiplies
actions by $s^2$.

\subsection{Proof of Proposition~\ref{prop:computation}}\label{sec:computation}

To compute symplectic homology of the unbounded domain
$V=V^{(a,b)}_k$, we first need to choose an increasing family
$V_C\subset V$ of bounded subdomains exhausting $V$. By definition of
the inverse limit it is up to us which sequence to choose, and we do
it  carefully to control the dynamics of the arising Hamiltonian
vector fields.

For $C>0$ set $g_C(t):=\max(-t^2, 3t^2-4C^2)$ and let $\wt
g_C:\R\to\R$ be a smoothing of the function $g_C$. We do the smoothing
such that for small $\eps>0$
$$
   \wt g_C(t) = \begin{cases}-t^2,&    |t|\leq C-\eps, \cr
   3t^2-4C^2, &  |t|\geq C+\eps,
\end{cases}
$$
and $\wt g_C$ has minima at the points $\pm  C$ with the minimal value $g_C(C)=-C^2+\eps$.
Moreover, we can choose $\wt g_C$ such that
\begin{equation}\label{eq:gC}
   \frac{t}{2}\wt g_C(t)\geq \wt g_C(t)\quad\text{for all }t\in\R.
\end{equation}
Consider the function $H_C(x,y):= \frac{x^2}{a^2}+\frac{\wt g_C(y)}{b^2}$
on $\R^2$, thus
$$
   H_C(x,y) = \begin{cases}
      \frac{x^2}{a^2}-\frac{y^2}{b^2}, & |y|\leq C-\eps, \cr
      -\frac{4C^2}{b^2}+\frac{x^2}{a^2}+\frac{3y^2}{b^2}, & |y|\geq C+\eps.
    \end{cases}
$$
The critical points of $H_C$ are a saddle point at the origin and two
minima at $(0,\pm C)$ with the minimal value $\frac{C^2+\eps}{b^2}$.
Note that $H_C(x,y)\geq\frac{x^2}{a^2}-\frac{y^2}{b^2}$.
Moreover, condition~\eqref{eq:gC} implies
\begin{equation}\label{eq:HC}
   \frac12\Bigl(x\frac{\p H_C}{\p x}+y\frac{\p H_C}{\p y}\Bigr)
   = \frac{x^2}{a^2}+\frac{y\wt g_C'(y)}{2b^2} \geq H_C(x,y).
\end{equation}

\begin{lemma}\label{lm:2d}
The Hamiltonian system of $H_C$ on the plane $(\R^2,dx\wedge dy)$ has
the folowing properties:
\begin{enumerate}
\item all its trajectories except the two homoclinic orbits at the origin are closed;
\item for every closed trajectory $\gamma$ we have $\int_\gamma xdy\geq 0$;
\item for every nonconstant closed trajectory $\gamma$ with
$H_C|_\gamma\geq-\frac{C^2}{4b^2}$ we have
$\int_\gamma xdy\geq C^2A_{a,b}$, with a constant $A_{a,b}$ depending only on $a$ and $b$.
\end{enumerate}
\end{lemma}

\begin{proof}
The zero level set of $H_C$ forms a figure eight consisting of the two
homoclinic orbits at the origin and enclosing the two
minima. Assertions (i) and (ii) follow from this picture.
For (iii), consider the closed trajectories $\gamma_\pm$ of value
$H_C|_{\gamma_\pm}=-\frac{C^2}{4b^2}$
enclosing the points $(0,\pm C)$. Let $V_\pm$ be the region bounded
by $\gamma_\pm$. Then $V_\pm\cap\{\pm y\geq C+\eps\}=E\cap\{\pm y\geq
C+\eps\}$ for the ellipse $E=\{\frac{x^2}{a^2}+\frac{3y^2}{b^2}\leq
\frac{15 C^2}{4b^2}\}$, and rescaling by $C$ shows that the area of
$V_\pm\cap\{\pm y\geq C+\eps\}$ equals $C^2A_{a,b}$ with a constant
$A_{a,b}$ depending only on $a$ and $b$. This proves
$\int_{\gamma_\pm} xdy\geq C^2A_{a,b}$, and the area of each
nonconstant closed trajectory $\gamma$ with $H_C|_\gamma>-\frac{C^2}{4b^2}$
is larger than this one.
\end{proof}

Define now the open domain
$$
   V_C := \{G < 1\}\subset\R^{2n}
$$
with the Hamiltonian
$$
   G(x_1,\dots, x_n, y_1,\dots, y_n) :=
   \sum_1^{n-k}\frac{x_j^2+y_j^2}{a^2} + \sum_{n-k+1}^n H_C(x_j,y_j).
$$

\begin{lemma}\label{lem:VC}
$V_C$ is a bounded domain contained in $V_k^{a,b}$ with smooth restricted
contact type boundary $\p V_C$. The closed characteristics on $\p V_C$
fall into two groups:
\begin{enumerate}
\item closed characteristics on the sphere $S$ of radius $a$ in the subspace
$\R^{2(n-k)}\subset\R^{2n}$, of actions $k\pi a^2$ for $k\in\N$;
\item all other closed characteristics have actions $\geq C^2B_{a,b}$,
  with a constact $B_{a,b}$ depending only on $a$ and $b$.
\end{enumerate}
\end{lemma}

\begin{proof}
$V_C$ is a bounded because $G$ is exhausting, and $V_C\subset
  V_k^{a,b}$ follows from $H_C(x,y)\geq\frac{x^2}{a^2}-\frac{y^2}{b^2}$.
Its boundary is of restricted contact type because it is transverse to
the Liouville vector field $Z=\frac12\sum_{1}^n(x_i\frac{\p}{\p
  x_i}+y_i\frac{\p}{\p y_i})$, which in turn follows from~\eqref{eq:HC} by
computation at points of $\p V_C=\{G=1\}$:
\begin{align*}
   Z\cdot G
   &= \sum_1^{n-k}\frac{x_j^2+y_j^2}{a^2} + \sum_{n-k+1}^n
   \frac12\Bigl(x_j\frac{\p H_C}{\p x_j}+y_j\frac{\p H_C}{\p y_j}\Bigr) \cr
   &= 1 + \sum_{n-k+1}^n\left(
   \frac12\Bigl(x_j\frac{\p H_C}{\p x_j}+y_j\frac{\p H_C}{\p y_j}\Bigr)-H_C(x_j,y_j)\right)
   \geq 1.
\end{align*}
Let us study the periodic orbits on $\p V_C$. First, we observe that the corresponding Hamiltonian system with the Hamiltonian $G$ is completely integrable and has integrals
$$
   G_j(x,y):=
   \begin{cases}\frac{x_j^2+y_j^2}{a^2},& j=1,\dots, n-k,\\
   H_C(x_j,y_j),& j=n-k+1,\dots, n.
   \end{cases}
$$
Hence the simple periodic orbits are given by equations $G_j=c_j$ for
$j=1,\dots,n$, with $\sum\limits_1^n c_j=1$.

Let us make an accounting of periodic orbits.
First of all we have orbits
$\{x_{n-k+1}=y_{n-k+1}=\dots x_{n}=y_{n}=0,\; G_j=c_j\geq 0,
j=1,\dots, n-k\}$ with $\sum\limits_1^{n-k} c_j=1$ which foliate the
sphere $S$ of radius $a$ in the subspace
$\R^{2(n-k)}\subset\R^{2n}$. These orbits and their multiples
correspond to group (i) in the lemma and their actions are $k\pi a^2$
for $k\in\N$.

Consider now a simple periodic orbit $\gamma$ which is not in group (i).
Note that $\gamma$ is a product of periodic orbits $\gamma_j$ for the
Hamiltonians $G_j$, and each $\gamma_j$ has nonnegative action by Lemma~\ref{lm:2d}(ii).
By assumption, at least one of the orbits $\gamma_j$,
$j=n-k+1,\dots,n$, is not the constant orbit at the origin.

If at least one of the constants $c_j$, $j=n-k+1,\dots,n$ is positive, then
the action of the orbit $\gamma_j$, and hence of $\gamma$, is $\geq C^2A_{a,b}$ by Lemma~\ref{lm:2d}(iii).
Otherwise, set $\delta:=-\sum\limits_{n-k+1}^n c_j>0$. Then
$\sum\limits_{1}^{n-k} c_j=1+\delta$, and therefore
$\sum\limits_1^{n-k}\int_{\gamma_j}x_jdy_j=\pi (1+\delta)a^2$.
If $\delta>\frac{C^2}{4b^2}$, then the action of the orbit $\gamma$ is $>\pi\Bigl(1+\frac{C^2}{4b^2}\Bigr)a^2$.
Otherwise, all the constants $c_j$ for $j=n-k+1,\dots,n$ satisfy
$c_j\geq-\frac{C^2}{4b^2}$. By assumption, at least one of the corresponding orbits
$\gamma_j$ is nonconstant, so by Lemma~\ref{lm:2d}(iii) the action of
this $\gamma_j$, and hence of $\gamma$, is $\geq C^2A_{a,b}$.
This proves Lemma~\ref{lem:VC}.
\end{proof}


\subsection*{Deformation to a split Hamiltonian}
We write $G=F_1+F_2$ with the Hamiltonians
$$
   F_1:=\sum_1^{n-k}G_j:\R^{2(n-k)}\to\R,\qquad F_2:=\sum_{n-k+1}^{n}G_j:\R^{2k}\to\R.
$$
Recall that $F_1,F_2$ have the following properties:
\begin{enumerate}
\item $F_1$ and $F_2$ are exhausting with $F_1\geq
  -(n-k)\frac{C^2+\eps}{b^2}$ and $F_2\geq 0$;
\item $Z_i\cdot F_i\geq F_i$ for the respective Liouville fields $Z_i$;
\item all periodic orbits of $F_i$ have action $\geq0$;
\item all periods of nonconstant periodic orbits of $F_i$ are bounded
  below by some $\delta>0$;
\item all second partial derivatives of $F_i$ are uniformly bounded;
\item the (non-Hamiltonian) action of each $k$-fold covered periodic
  orbit $z_2$ of $F_2$ satisfies $\AA(z_2)=\pi ka^2F_2(z_2)\geq \pi a^2F_2(z_2)$.
\end{enumerate}
Consider a family of Hamiltonians $H_s:\R^{2n}\to\R$, $s\in\R$, of the form
$$
   H_s(z_1,z_2) = h_s\bigl(F_1(z_1),F_2(z_2)\bigr)
$$
with a smooth family of function $h_s:\R^2\to\R$ satisfying the following properties:
\begin{enumerate}
\item $h_s$ is locally constant in $s$ outside a compact subset of $\R$;
\item outside a compact subset of $\R^2$ we have $0<\frac{\p h_s}{\p
  F_1}<\delta$ or $0<\frac{\p h_s}{\p F_2}<\delta$ (or both);
\item all second partial derivatives of the function $\R^3\to\R$,
  $(s,F_1,F_2)\mapsto h_s(F_1,F_2)$ are uniformly bounded.
\end{enumerate}

\begin{lemma}\label{lem:max-principle}
For $H_s$ as above all $1$-periodic orbits are contained in a compact
set, and for each $c>0$ the Floer homology $FH^{(0,c)}(H_s)$ is
well-defined and independent of $s$.
\end{lemma}

\begin{proof}
For each $s$ the $1$-periodic orbits of $H_s$ are of the form $z=(z_1,z_2)$, where the
$z_i$ satisfy $\dot z_i=\frac{\p h_s}{\p F_i}X_{F_i}(z_i)$. Hence
$F_1,F_2$ are constant along $z$ and
$z_i(t)=\gamma_i\Bigl(\frac{\p h_s}{\p F_i}t\Bigr)$ for periodic orbits
$\gamma_i$ of $X_{F_i}$ of period $\frac{\p h_s}{\p F_i}$. Therefore,
conditions (iv) on $F_i$ and (ii) on $h_s$ imply that all $1$-periodic
orbits of $X_{H_s}$ are contained in a compact set.

Next, let $u:\R\times S^1\to\R^{2n}$ be a Floer cylinder connecting
$1$-periodic orbits. It satisfies
$$
   u_s+iu_t+\nabla H_s(u)=0,
$$
where $u_s,u_t$ denotes the partial derivatives with respect to the
coordinates $(s,t)\in\R\times S^1$. The bounds on the second
derivatives of $F_i$ and $h$ yield uniform bounds $|D^2H_s|\leq A$ on
the Hessian of $H_s$ and $|\nabla\p_sH_s(u)|\leq B|u|$ on the gradient of
the $s$-derivative $\p_sH_s$. Using this, a standard computation shows that the
function $\rho(s,t):=|u(s,t)|^2$ satisfies the estimate
\begin{align*}
   \Delta\rho
   &= |u_s|^2+|u_t|^2+\la u,iD^2H(u)u_t-D^2H(u)u_s+\nabla\p_sH_s(u)\ra \cr
   &\geq |u_s|^2+|u_t|^2-A|u|(|u_s|+|u_t|) - B|u|^2
   \geq -\Bigl(\frac12A^2+B\Bigr)\rho.
\end{align*}
By an argument in~\cite{Cie94}, this estimate implies that
Floer cylinders for the $s$-dependent Hamiltonian $H_s$ remain in a
compact set, hence the Floer homology of $H_s$ is well-defined and
independent of $s$.
\end{proof}

Pick a nondecreasing function $\phi:\R\to\R$ satisfying
$\phi(t)=\delta t/2$ for $t\geq 0$ and $\phi(t)\equiv -m$ for
$t\leq\delta$ with some large constant $m$.
For $s\in[0,1]$ consider the Hamiltonian
$$
   H_s = h_s(F_1,F_2) := (1-s)\phi(F_1+F_2-1)+s\phi(F_1)+s\phi(F_2-1).
$$

\begin{lemma}\label{lem:Hs}
For $\phi$ as above with $m>c$ each $1$-periodic orbit $(z_1,z_2)$ of $H_s$
with action in the interval $(0,c)$ satisfies $z_1\equiv0$, and $z_2$ is
a $1$-periodic orbit of $\phi(F_2-1)$ of action $\pi a^2k$ for
$k=1,\dots,\lfloor\frac{c}{\pi a^2}\rfloor$.
\end{lemma}

\begin{proof}
Consider a $1$-periodic orbit $z=(z_1,z_2)$ of $H_s$ with action in the
interval $(0,c)$. Its components satisfies the equations
\begin{align*}
   \dot z_1 &= \Bigl((1-s)\phi'(F_1+F_2-1)+s\phi'(F_1)\Bigr)X_{F_1}(z_1),\cr
   \dot z_2 &= \Bigl((1-s)\phi'(F_1+F_2-1)+s\phi'(F_2-1)\Bigr)X_{F_2}(z_2).
\end{align*}
We distinguish three cases.

Case 1: $z_1$ is not constant. Then by Lemma~\ref{lm:2d}, for the action
to be below $c$, each nonconstant component of $z_1$ must have value
$H_C\leq-\frac{C^2}{4b^2}$. Since each constant component has value
$\leq 0$ and at least one component is nonconstant, we deduce
$$
   F_1(z_1)\leq -\frac{C^2}{4b^2}.
$$
It follows that $\phi(F_1)=-m$ and $\phi'(F_1)=0$, so $z_1$ satisfies
the equation $\dot z_1 = (1-s)\phi'(F_1+F_2-1)X_{F_1}(z_1)$. Since
$\dot z_1\neq 0$, we must have $s<1$ and $F_1+F_2-1>-\delta$. Together
with the preceding displayed equation this yields
$F_2(z_2)>1-\delta-F_1(z_1)\geq 1-\delta+\frac{C^2}{4b^2}$, which in
view of property (vi) of $F_2$ implies $\AA(a_2)>\pi a^2\Bigl(1-\delta+\frac{C^2}{4b^2}\Bigr)>c$.
So Case 1 cannot occur.

Case 2: $z_1$ is constant but not all its component are zero. We rule
this out by distinguishing several cases.

(i) If $X_{F_1}=0$, then each components of $z_1$ is a critical point
of $H_C$, with at least one of them being nonzero. Since $H_C(0)=0$
and $H_C=$ at the minima we conclude
$F_1(z_1)\leq -\frac{C^2+\eps}{b^2}\leq -\frac{C^2}{4b^2}$, which is
ruled out as in Case 1.

(ii) If $X_{F_1}\neq0$ and $0<s<1$, then we must have
$\phi'(F_1+F_2-1)=\phi'(F_1)=0$, hence $\phi(F_1+F_2-1)=\phi(F_1)=-m$.
Then $z_2$ satisfies $\dot z_2=s\phi'(F_2-1)X_{F_2}(z_2)$, so by
the choice of $\phi$ it can only be $1$-periodic if $\phi(F_2-1)\leq 0$.
Thus $H_s(z)\leq -m$ and the Hamiltonian action of $z$ satisfies
$\AA_{H_s}(z)\geq m>c$.

(iii) If $X_{F_1}\neq0$ and $s=0$, then $\phi'(F_1+F_2-1)=0$, hence
$H_0(z)=\phi(F_1+F_2-1)=-m$ and again $\AA_{H_0}(z)\geq m>c$.

(iv) If $X_{F_1}\neq0$ and $s=1$, then $\phi'(F_1)=0$, hence
$\phi(F_1)=-m$. Then $\dot z_2=\phi'(F_2-1)X_{F_2}(z_2)$ implies
$\phi(F_2-1)\leq 0$, thus $H_1(z)\leq -m$ and again $\AA_{H_1}(z)\geq m>c$.

Case 3: $z_1\equiv0$. Then $F_1(z_1)=0$ and $z_2$ is of the form
described in the lemma.
\end{proof}

\begin{proof}[Proof of Proposition~\ref{prop:computation}]
Lemma~\ref{lem:max-principle} and Lemma~\ref{lem:Hs} together imply
(after replacing $H_s$ by $H_{\sigma(s)}$ for a nondecreasing function
$\sigma:\R\to[0,1]$ which equals $0$ for $s\leq0$ and $1$ for $s\geq1$)
that the Floer homology $FH^{(0,c)}(H_s)$ is independent of $s\in[0,1]$.
By definition, the Hamiltonian $H_0(z)=\phi\bigl(G(z)-1\bigr)$
computes the symplectic homology of $V_C$,
$$
   FH^{(0,c)}(H_0)\cong SH^{(0,c)}(V_C).
$$
The Hamiltonian
$$
   H_1(z_1,z_2) = \phi\bigl(F_1(z_1)\bigr)+\phi\bigl(F_2(z_2)-1\bigr)
$$
is split, as well as the corresponding Floer equation.
It follows that all its Floer cylinders are contained in the subspace $\R^{2(n-k)}$,
so $H_1$ computes the symplectic homology of the ball $B_a^{2(n-k)}$
of radius $a$ in $\R^{2(n-k)}$. This symplectic homology is computed
in~\cite{FHW94}, up to an index shift by $1$ due to our different
conventions, to be
$$
   SH_j^{(0,c)}(B_a^{2(n-k)}) \cong\begin{cases}
      \Z_2, & c>\pi a^2\text{ and }j=n-k, \\
      \Z_2, & c>\pi a^2\text{ and }j=(n-k)\Bigl(2\lfloor\frac{c}{\pi a^2}\rfloor-1\Bigr)-1, \\
      0, & \text{otherwise}.
   \end{cases}
$$
This proves part (a) of Proposition~\ref{prop:computation}. Parts (b)
and (c) follow by similar arguments.
\end{proof}

\section{Geometry of the dominating positive cone}\label{section-geometry}

Let $\g^+$ be the maximal dominating cone of an open contact manifold $(U,\xi)$.
The group $\R_+$ acts on $\Theta:= \g^+/\sim$ by multiplication.
We assume that $\g^+$ is orderable up to conjugation and consider the
binary relation $\preceq$ on $\Theta$ from Section~\ref{sec:partial-order}.
For a pair of classes $f,h \in \Theta$ define
\begin{equation}\label{eq-rho}
   \rho(f,h) := \inf\{s>0\mid f \preceq sh\}\;.
\end{equation}
The fact that $\g^+$ is dominating implies that the set on the right
hand side is nonempty,
and orderability up to conjugation means that {\it there exist} $a,b \in \Theta$ such that $\rho(a,b) \neq 0$.
Furthermore, clearly we have sub-multiplicativity
\begin{equation}\label{eq-rho-1}
\rho(f,h) \leq \rho(f,g)\rho(g,h) \;\; \text{for all }f,g,h \in \Theta\;.
\end{equation}
Observe that $\rho(h,h) \geq 1$ for all $h \in \Theta$ by Lemma \ref{lem:non-orderable-g} (a). On the other hand, obviously $\rho(h,h) \leq 1$ and hence $\rho(h,h)=1$.

We claim that $\rho(f,g) \neq 0$ {\it for all} $f,g \in \Theta$. Indeed, take $a,b$ with $\rho(a,b)\neq 0$
and write, by sub-multiplicativity,
$$0< \rho(a,b) \leq \rho(a,f)\rho(f,g)\rho(g,b)\;,$$
yielding $\rho(f,g) \neq 0$. The claim follows.

Define now a function $d: \Theta \times \Theta \to \R$
by
\begin{equation}\label{eq-d}
d(g,h) = \max (|\log \rho(g,h)|,|\log\rho(h,g)|)\;.
\end{equation}
The above discussion shows that $d$ is a pseudo-metric on $\Theta$: it is symmetric,
nonnegative and satisfies the triangle inequality. It is unknown whether $d$ is a genuine distance on $\Theta$ (and sounds unlikely that it is). Introduce the equivalence relation $\approx$ on $\Theta$
by $f \approx g$ whenever $d(f,g)=0$. This relation measures the deviation of $d$ from a genuine metric.
Interestingly enough, it also measures the deviation of the binary relation $\preceq$ from a genuine
partial order. Indeed, if $f \preceq g$ and $g \preceq f$, then we have $\rho(f,g) \leq 1$ and
$\rho(g,f) \leq 1$. By  \eqref{eq-rho-1},
$$1 = \rho(f,f) \leq \rho(f,g)\rho(g,f)\;,$$
and hence
$$\rho(f,g)=\rho(g,f) =1\;.$$ This yields $d(f,g)=0$ and hence $f \approx g$. Denote
$\Xi:= \Theta/\approx$, and note that the pseudo-metric $d$ descends
to $\Xi$ as a genuine metric $D$.
What about the partial order? Define a relation $\ll$ on  $\Xi$ as follows:
$p \ll q$ if there exist $f,g \in \Theta$ and a sequence of positive numbers $\eps_i \to 0$ so that
$p=[f]$, $q=[g]$ and $f \preceq (1+\eps_i)g$ for all $i$. (The
$\eps_i>0$ are needed because of the infimum in~\eqref{eq-rho}). Note that this is a transitive and reflexive relation. We claim that {\it $\ll$ is a genuine partial order.}  Indeed if $p \ll q$ and $q \ll p$, then $D(p,q)=0$ and hence $p=q$. It would be interesting to explore the geometry of $\Xi$
with respect to $D$.

\begin{rem}\label{rem-d-cbm} {\rm The quantity $\rho$ and the pseudo-metric $d$ has cousins in the earlier literature. On the one hand, they can be considered Lie algebra counterparts of the relative growth between positive
contactomorphisms, and the corresponding pseudo-metric, respectively, studied in \cite{EP}. On the other hand,
each compactly supported non-negative function $H$ on $U$ defines a ``contact form" $\alpha/H$ on $U$,
where the quotation marks stand for the fact that this form could be infinite and certainly is infinite
outside a compact set. With this language, the adjoint action of contactomorphisms on functions corresponds
to the action of contactomorphisms on contact forms, and the metric $d$ is an analogue, for open manifolds,
of the {\it contact Banach-Mazur distance} introduced by Yaron Ostrover and the third-named author and discussed in
\cite{PZ,SZ} in the context of closed contact manifolds. }
\end{rem}

\begin{example}\label{exam-d-uk} Consider $(S^{2n-1} \setminus
  \Pi_k,\xi_{st})$ as above with $k<n$.
The function $w: \g^+ \to (0,\infty)$ from Theorem~\ref{thm:order-algebra}
descends to a function $\wt{w}: \Theta \to (0,\infty)$ with the following properties:
\begin{itemize}
\item[{(i)}] $h \preceq f \Rightarrow \wt{w}(h) \geq \wt{w}(f)$;
\item[{(ii)}] $ \wt{w}(sh)= s^{-2}\wt{w}(h)\;\; \text{for all }s >0, h \in \Theta\;.$
\end{itemize}
These properties readily yield the following inequality:
\begin{equation} \label{eq-dw}
d(f,h) \geq \frac{1}{2}\big{|}\log \frac{\wt{w}(f)}{\wt{w}(h)}\big{|}\;.
\end{equation}
For instance, this shows that $d(g,sg)= |\log s|$, and in particular
the restriction of the pseudo-metric $d$ to each orbit of the
$\R_+$-action on $\Theta$ is isometric to the Euclidean line. It would
be interesting to explore whether $\Theta$ admits a quasi-isometric
embedding of the Euclidean $\R^N$ for $N \geq 2$. Let us mention also
that the above conclusions continue to hold verbatim for the metric space $(\Xi, D)$.
\end{example}

 \end{document}